\documentclass[10pt]{article}
\usepackage{amssymb,amsmath,textcomp, amsthm}
\usepackage{cases}
\newtheorem{theorem}{Theorem}[section]

\newtheorem{remark}{Remark}[section]

\newcommand{\RR}{\mathbb{R}}
\newcommand{\NN}{\mathbb{N}}

\newcommand{\abs}[1]{\lvert#1\rvert}

\begin{document}
\title{Asymptotic behavior of solutions to a $k$-Hessian evolution equation}

\author{Justino S\'anchez}
\date{}
\maketitle
\begin{center}
Departamento de Matem\'{a}ticas, Universidad de La Serena\\
 Avenida Cisternas 1200, La Serena, Chile.
\\email: jsanchez@userena.cl
\end{center}

\maketitle
\begin{abstract}
We study the long-time behavior of solutions of the $k$-Hessian evolution equation $u_t=S_{k}(D^2 u)$, posed on a bounded domain of the $n$-dimensional space with homogeneous boundary conditions. To this end, we construct a separable solution and we show that the long-time behavior of $u$ is precisely described by this special solution. Further, we initiate the study of that dynamic phenomenon on the entire space, providing a new class of explicit and radially symmetric self-similar solutions that we call $k$-{\it Barenblatt solutions}. These solutions present some common properties as those of well-known Barenblatt solutions for the porous media equation and the $p$-Laplacian equation. It is known that self-similar solutions are important in describing the intermediate asymptotic behavior of general solutions.
\end{abstract}

\section{Introduction} 
Let $n\in\NN, \,n\geq 2$, and let $k$ be an integer $1\leq k\leq n$. Let $\Omega\subset\RR^n$ be a bounded, strictly $(k-1)$-convex domain with boundary $\partial\Omega\in C^2$, and $u(t,x),\, t\in [0,\infty),\, x\in\overline{\Omega}$, a solution of the initial boundary value problem
\begin{equation}\label{parak-Hessian}
\begin{cases}
u_t=S_{k}(D^2 u)\;\; \mbox{in}\;\; (0,\infty)\times\Omega,\\
u(t,x)\, \mbox{ is }\, k\mbox{-admissible for each}\; t\geq 0,\\
u(t,x)=0\;\; \mbox{on}\;\; [0,\infty)\times\partial\Omega,
\end{cases}
\end{equation}
where $D^2u=\left(\frac{\partial^2u}{\partial x_i \partial x_j}\right)$ is the Hessian matrix of $u$ and $S_{k}(D^2 u)$ denotes the $k$-Hessian operator of $u$. 

Our central interest in this paper is to study the long-time asymptotic behavior of classical solutions $u(t,x)$ of Problem \eqref{parak-Hessian} with $k>1$, and to derive estimates involving $u(t,x)$ and a special solution of separate variables. In particular, sharp estimates for the rate at which $u(t,x)\rightarrow 0$ as $t\rightarrow\infty$ in terms of this solution are given. To this end, we construct a solution of the form $u(t,x)=T(t)\theta (x)$, where $T(t)\in C^\infty[0,\infty)$ and the profile $\theta(x)$ belongs to a suitable class of functions that ensure ellipticity of $S_{k}(D^2 u)$. Using Hessian Sobolev inequalities and the comparison principle for elliptic equations, we show that the $\sup_{\Omega}|\theta(x)|$ admits an estimate depending only on $n,k$ and $\Omega$. An interesting geometric property arises when $\Omega$ is a ball, where $(1+t)^{1/(k-1)}u(t,x)$ asymptotically approaches a rotationally symmetric state.

In order to stating our assumptions and results more precisely, we will briefly discuss the main features of $k$-Hessian operators beginning with its definition. For a twice-differentiable function $u$ defined on a domain $\Omega\subset\RR^n$, the {\it $k$-Hessian operator} $(k=1,...,n)$ is defined by the formula 
\[
S_k(D^2u)=\sigma_k(\Lambda)=\sum_{1\leq i_1<...<i_k\leq n}\lambda_{i_1}...\lambda_{i_k},
\] 
where $\Lambda=\Lambda(D^2u):=(\lambda_1,...,\lambda_n)$, the $\lambda$'s are the eigenvalues of $D^2u$ and $\sigma_k$ is the $k$-th elementary symmetric function. Equivalently, $S_k(D^2u)$ is the sum of the $k$-th principal minors of the Hessian matrix. See, {\it e.g.}, \cite{Wang94, Wang09}. These operators form an important class of nonlinear second order operators that contains, as the most relevant examples, the Laplace operator $S_1(D^2u)=\Delta u$  and the Monge-Amp\`{e}re operator $S_n(D^2u)=$ det $D^2u$. They are fully nonlinear when $k>1$. The study of $k$-Hessian equations has many applications in geometry, optimization theory and in other related fields. See \cite{Wang09}. There exists a large literature about existence, regularity and qualitative properties of solutions for the $k$-Hessian equations, starting with the seminal work \cite{CaNS85}. See {\it e.g.} \cite{ChWa01, Jacobsen99, Jacobsen04, JaSc02, Tso89, Tso90} and the references therein.

We point out that the $k$-Hessian operators are of divergence form, $k$-homogeneous and also invariant under rotations. Further, they are not elliptic in general. Thus, in order to ensure ellipticity of $S_k(D^2u)$, one should look for solutions $u$ in the class 
\begin{equation*}
\Phi^k(\Omega)=\{u\in C^2(\Omega)\cap C(\overline{\Omega}): \sigma_{j}(\Lambda(D^2u))>0\;\;\mbox{in}\;\;\Omega,\, j=1,...,k\}.
\end{equation*}
The functions in $\Phi^k(\Omega)$ are called $k$-{\it admissible} (or {\it uniformly $k$-convex}) on $\Omega$. Note that $n$-admissible functions are convex in the usual sense. The space of $k$-admissible functions is strictly larger than the space of convex functions for $1\leq k\leq n-1$ (see \cite{CaNS85} for the proof). Denote by $\Phi_0^k(\Omega)$ the set of functions in $\Phi^k(\Omega)$ that vanish on the boundary $\partial\Omega$.
Observe that the functions in $\Phi_0^k(\Omega)$ are subharmonic and, by the maximum principle, negative on $\Omega$. See, {\it e.g.}, \cite{Wang94}. 
The $k$-Hessian operator defined on $\Phi_0^k(\Omega)$ imposes certain restrictions on the geometry of $\Omega$. More precisely, those domains $\Omega$ whose boundary $\partial\Omega$ satisfies the inequality $\sigma_j(\kappa_1, ..., \kappa_{n-1})>0$ for each $j=1,...,k$, where $\kappa_1, ..., \kappa_{n-1}$ denote the principal curvatures of $\partial\Omega$ relative to the interior normal, are called {\it strictly $k$-convex} (or {\it uniformly $k$-convex}). 
A typical example of a domain $\Omega$ for which the above inequality holds is a ball. Notice that uniformly $(n-1)$-convex domains are strictly convex in the classical sense. For more details we refer the interested reader to \cite{Wang09}.

An interesting case where \eqref{parak-Hessian} arises with $k=n,\, \Omega$ a bounded strictly convex domain in $\RR^n,\, n\geq 2$, with smooth boundary $\partial\Omega$ and $u(t,x)$ is strictly convex (downward) for each $t\geq 0$, is in the study of evolution problems for nonparametric surfaces with speed depending on curvature \cite{Oliker91}. By first searching for solutions of separate variables, $u(t,x)=\varphi(t)\psi(x)$, the author in \cite{Oliker91} shows that the function $\psi$ should satisfy the Monge-Amp\`{e}re equation
\begin{equation}\label{psi}
\mbox{det}(D^2\psi)=-\frac{1}{n-1}\,\psi\;\; \mbox{in }\, \Omega,\, \psi=0\;\; \mbox{on }\; \partial\Omega, 
\end{equation}
where $\psi$ is strictly convex and negative on $\Omega$. The solution found from separation of variables is then used to describe the asymptotic behavior of solutions of a more general problem involving the Gauss curvature, which exhibits several intriguing geometric properties. In the present paper we show that the ideas in \cite{Oliker91} can be used to study Problem \eqref{parak-Hessian} with $k>1$. In fact, we extend some results of \cite{Oliker91} to the whole class of fully nonlinear $k$-Hessian operators, {\it i.e.}, the cases for $2\leq k\leq n$. In particular, we have the estimate
\begin{equation*}
\sup_{\Omega}\left|(1+t)^{1/(k-1)}u(t,x)-\theta(x)\right|\leq C(1+t)^{-1},
\end{equation*}
which holds for solutions of Problem \eqref{parak-Hessian}. Here the function $\theta(x)$ is the unique negative solution of the stationary problem
\begin{equation}\label{teta}
S_{k}(D^2\theta)=-\frac{1}{k-1}\,\theta\;\, \mbox{in}\;\, \Omega,\, \theta=0\,\, \mbox{ on}\;\,\partial\Omega.
\end{equation}

We point out that the proof of the existence of $\psi$ in \eqref{psi} was the most lengthy portion of \cite{Oliker91}, on the contrary, here, the proof of the existence and uniqueness of $\theta$ in \eqref{teta} can be deduce directly from Theorem 5.9 of \cite{Jacobsen99}. Concerning stationary problems, as far as we known, the first approach to study the class of $k$-Hessian operators using global bifurcation was done in \cite{Jacobsen99}. Recently \cite{Dai17}, an eigenvalue problem for $k$-Hessian equations was studied. Results on existence, non-existence, uniqueness and multiplicity of radially symmetric admissible solutions were also determined by the bifurcation method. 

Concerning stationary problems, as far as we know, the first approach for studying the class of $k$-Hessian operators using global bifurcation was done in \cite{Jacobsen99}. Recently, an eigenvalue problem for $k$-Hessian equations was discussed in \cite{Dai17}. Results on the existence, non-existence, uniqueness and multiplicity of radially symmetric admissible solutions were also determined by the bifurcation method in \cite{Dai17}. 

We should also mention \cite{IvFI18}, where some distinctions between the classical and the contemporary theory of second-order fully nonlinear parabolic and elliptic partial differential equations are established. New algebraic and geometric structures are investigated based on notions from the algebra of symmetric matrices and differential geometry. Among others, new results on the asymptotic behavior of the $m$-Hessian evolution operators acting on bounded domains are established. Such operators are defined by $E_m[u]:=-u_tT_{m-1}[u]+T_m[u],\, m=1,...,n$, where $T_m[u]$ is the $m$th-order trace of the Hessian matrix of $u$. See \cite{IvFI18} for the precise statements on these results. 

The last part of this paper is devoted to the study of self-similar solutions of a $k$-Hessian evolution equation posed in $\RR^n$. This study is a first step towards understanding important properties of the underlying equations which can be captured by these special solutions. We point out that there is a vast literature concerning evolution equations that generalize the standard heat equation. This literature addresses among others, the $p$-Laplacian equation, the porous medium equation and the space-fractional porous medium equation. See {\it e.g.} \cite{Bidaut-Veron09, FiWi08, FiWi16, GaVa04, Huang14, ISVa08, KaVa88, MaMe09, QuSo07, Wang93}. However, relatively little is known in the case of fully nonlinear parabolic equations on the entire space, which includes different parabolic analogues of the $k$-Hessian equation. A work in that direction was carried out in \cite{ArTr10}. There, the authors studied the long-time asymptotics of solutions of the uniformly parabolic equation
\begin{equation}\label{nonexplicit}
u_t+F(D^2u)=0\;\; \mbox{in }\; \RR^n\times \RR_+,
\end{equation}
where $F$ is a positively homogeneous operator, subject to the initial condition $u(x,0)=g(x)$, and where the function $g$ does not change sign and has a proper decay at infinity. Although the operator $F$ is not assumed to be rotationally invariant, as the $k$-Hessian is, it is assumed to be positively homogeneous of order one, {\it i.e.}, $F(\eta M)=\eta F(M)$ for all $n\times n$ real symmetric matrices $M$ and $\eta\geq 0$, a hypothesis that is satisfied by $[S_k(D^2u)]^{1/k}$ but not by $S_k(D^2u)$ if $k>1$. The basic existence, uniqueness and other properties of the self-similar profiles of \eqref{nonexplicit} were established in \cite{ArTr10}, without any explicit expressions for them. 

Concerning exact solutions of some nonlinear diffusion equations, in \cite{King90} new closed-form similarity solutions of $N$-dimensional radially symmetric equations were given, which are generalizations of the classical Barenblatt solutions. In \cite{ISVa08}, the authors study an explicit equivalence between radially symmetric solutions for two basic nonlinear degenerate diffusion equations, namely, the porous medium equation and the $p$-Laplacian equation. In particular, they derive the existence of new self-similar solutions for the evolution $p$-Laplacian equation. In \cite{Huang14} several one-parameter families of explicit self-similar solutions were constructed for the porous medium equations with fractional operators.\\

\noindent Notice that the equation in \eqref{parak-Hessian} for negative solutions is equivalent to equation
\begin{equation}\label{k-parabol}
u_t=(-1)^{k-1} S_k(D^2 u)
\end{equation}
for positive solutions, by the $k$-homogeneity of the $k$-Hessian operator. When $k>1$ in \eqref{k-parabol}, we have found an explicit family of positive self-similar solutions on $\RR^n$ with compact support in space for every fixed time. As these solutions have apparently not been mentioned in the literature, we describe them now: 
\begin{equation}\label{k-Baren}
U_C(t,x)=t^{-\alpha}\left(C-\gamma\left(\frac{\abs{x}}{t^\beta}\right)^2\right)_{+}^\frac{k}{k-1},
\end{equation}
where $(\cdot)_+$ denotes the positive part, $C>0$ is an arbitrary constant, and $\alpha, \beta$ and $\gamma$ have precise values, namely
\[
\alpha=\frac{n}{n(k-1)+2k},\;\;\;\; \beta=\frac{1}{n(k-1)+2k},\;\;\;\; \gamma=\frac{k-1}{2k}\left(\frac{\beta}{c_{n,k}}\right)^\frac{1}{k},\;\;\;\; c_{n,k}=\frac{1}{n}\binom{n}{k}.
\] 
Note that this family, whose elements we call {\it $k$-Barenblatt solutions}, is well defined for the full range of $k$-Hessian operators with $k>1$. Moreover, these solutions are similar to those known for the porous medium equation and the $p$-Laplacian equation as well. See, {\it e.g.}, \cite{ISVa08} and the references therein. We also note that the relation between the similarity exponents $\alpha$ and $\beta$,\, $\alpha=n\beta$, is an {\it a priori} condition that reflects the mass conservation of these special solutions.\\

The paper is organized as follows. Section \ref{Mainresults} is devoted to stating our main results, namely Theorem \ref{separablesolution}, Theorem \ref{asymptotic} and Theorem \ref{radialasymptotic}, which are proven in Section \ref{proofs}. Finally, in Section \ref{k-Barenblatt} we derive the family of self-similar solutions given in \eqref{k-Baren} and we present some of its properties.

\section{Main results}\label{Mainresults}
Firstly, we claim that any solution of Problem \eqref{parak-Hessian} is global, that is, is defined at all times. Indeed, it follows from \eqref{parak-Hessian} that, for each $x\in\Omega$, $u(t,x)$ is a nondecreasing function of $t$. Then, $u(t,x)\geq u(t_1,x)$ for $t\geq t_1$. In particular, $0\geq u(t,x)\geq u(0,x)$. Consequently, from the dichotomy between global existence and finite time blow-up we conclude that $u(t,x)$ is global.

Secondly, we search for separable solutions of Problem \eqref{parak-Hessian} of the form 
\begin{equation}\label{separable}
u(t,x)=T(t)\theta (x),
\end{equation}
where $T(t)\in C^\infty[0,\infty)$ and $\theta(x)\in C^2(\Omega)\cap C(\overline{\Omega})$. Since, $u(0,x)=T(0)\theta (x)$ vanishes on $\partial\Omega$ and must be $k$-admissible, either $-\theta$ or $\theta$ is $k$-admissible, depending on the sign of $T(0)$. More precisely, either $T(0)<0$ and $-\theta$ is $k$-admissible or $T(0)>0$ and $\theta<0$ and $k$-admissible. Using the $k$-homogeneity of the $k$-Hessian operator, both cases are equivalent and we always assume below that we are dealing with the latter case.\\

Inserting \eqref{separable} into the equation in \eqref{parak-Hessian}, and taking into account that $S_{k}(D^2\theta)>0$ on $\Omega$, we obtain
\begin{equation*}\label{sepaequ}
\frac{T'}{T^k}= \frac{S_{k}(D^2\theta)}{\theta}=\lambda=\mbox{const}\leq 0. 
\end{equation*}
Hence,
\begin{equation}\label{Te}
T(t)=\left[T^{1-k}(0)-(k-1)\lambda t\right]^\frac{1}{1-k},\; t\geq 0,
\end{equation}
\begin{equation}\label{theta}
S_{k}(D^2\theta)=\lambda\theta\;\; \mbox{in}\;\; \Omega\;\; \mbox{and}\;\; \theta=0\;\; \mbox{on}\;\; \partial\Omega.
\end{equation}
The case $\lambda=0$ is uninteresting since it leads to the trivial solution. In fact, in this case $T(t)=\mbox{const}$ and it follows from \eqref{parak-Hessian} that either $S_{k}(D^2\theta)\equiv 0$ or $T(t)\equiv 0$. In the former case an easy argument shows that $\theta\equiv 0$ and, thus, in either case $u(t,x)\equiv 0$. In the following we consider only the case $\lambda<0$ and establish the existence of separable solutions of \eqref{parak-Hessian} by proving the following result.

\begin{theorem}\label{separablesolution}
Let $\Omega$ be a bounded strictly $(k-1)$-convex domain with boundary $\partial\Omega\in C^2$. Then problem \eqref{parak-Hessian} admits a separable solution of the form
\[
u_s(t,x)=(1+t)^{-1/(k-1)}\theta(x),\;\; t\geq 0,\;\;x\in\overline{\Omega},
\]
where $\theta$ is the unique solution in $\Phi_0^k(\Omega)$  
of the nonlinear problem
\begin{equation}\label{k-solution}
\begin{cases}
S_{k}(D^2\theta)=-\frac{1}{k-1}\,\theta\;\, \mbox{in}\;\, \Omega,\\
\theta<0\;\, \mbox{ in}\;\, \Omega,\\
\theta=0\;\, \mbox{ on}\;\,\partial\Omega.
\end{cases}
\end{equation}

Further, $\sup_{\Omega}|\theta(x)|$ admits an estimate depending only on $n,k$ and $\Omega$. If $\tilde{u}_s(t,x)=T(t)\tilde{\theta}(x)$ is an arbitrary separable solution of \eqref{parak-Hessian}, then there exists a unique constant $c>0$ such that $\tilde{\theta}(x)=c\,\theta(x)$ and 
\begin{equation}\label{formula}
\tilde{u}_s(t,x)=u_s(t,x)\left\{\frac{1+t}{[c\,T(0)]^{1-k}+t}\right\}^{1/(k-1)}.
\end{equation}
\end{theorem}
 
Our main result is the following. 
 
\begin{theorem}\label{asymptotic}
Let $u(t,x)$ be a solution of the problem
\begin{equation}\label{parabolick-Hessian}
\begin{cases}
u_t=S_{k}(D^2 u)\; \mbox{in}\; (0,\infty)\times\Omega,\\
u(t,x)\, \mbox{ is }\; k\mbox{-admissible for each}\; t\geq 0,\\
u(t,x)=0\; \mbox{ on}\; [0,\infty)\times\partial\Omega.
\end{cases}
\end{equation}
Then, for all $t\geq 0$,
\begin{equation}\label{bounded}
\sup_{\Omega}\left|(1+t)^{1/(k-1)}u(t,x)-\theta(x)\right|\leq C(1+t)^{-1},
\end{equation}
where $C$ is a positive constant depending on the dimension $n,\,\Omega$, and $u(0,x)$.
Further, estimate \eqref{bounded} is the best possible.
\end{theorem} 

Theorem \ref{asymptotic} has the following interesting geometric consequence. 
\begin{theorem}\label{radialasymptotic}
Suppose that the domain $\Omega$ is a ball in $\RR^n$ and $u$ is a solution of \eqref{parabolick-Hessian}. Then:
\[
\left|(1+t)^{1/(k-1)}u(t,x)-\theta(|x|)\right|\longrightarrow 0\,\mbox{ as}\,\, t\longrightarrow\infty\, \mbox{ uniformly on}\;\, \overline{\Omega}.
\]
In others words, $u(t,x)$ asymptotically becomes radially symmetric regardless of its initial shape.
\end{theorem}

\section{Proofs of the main results}\label{proofs}
\begin{proof}[Proof of Theorem \ref{separablesolution}]
The existence and the uniqueness of solutions to Problem \eqref{k-solution} follows by taking $p=1$ and $\delta=-\lambda=1/(k-1)$ in Theorem 5.9 of \cite{Jacobsen99}. The fact that $u_s(t,x)=(1+t)^{-1/(k-1)}\theta(x),\; t\geq 0,\;x\in\overline{\Omega}$, satisfies \eqref{parak-Hessian} is checked by direct substitution.

Next we prove that $\|\theta\|_{L^\infty(\Omega)}=\sup_\Omega|\theta(x)|$ can be estimated by a constant depending only on the dimension $n,\, k$ and the domain $\Omega$. To this end, for $u\in\Phi_0^k(\Omega)$, let $H_k(u)=-\int_\Omega uS_{k}(D^2u)\, dx$. It is well-known that $\|u\|_{\Phi_0^k(\Omega)}=[H_k(u)]^{1/(k+1)}$ defines a norm on the set $\Phi_0^k(\Omega)$ (see Theorem 5.1 in \cite{Wang94}). Furthermore, Hessian Sobolev inequalities involving this norm have been established (see Theorem 5.2 in \cite{Wang94} and Theorem 5.1 in \cite{Wang09}). We will use these inequalities to estimate the $L^\infty$ norm of $\theta$.\\

\noindent We consider three cases according to the different values of $k,\, 2\leq k\leq n$.
\begin{itemize}
\item[a)] Case $n/2<k\leq n$. By Theorem 5.1-$(iii)$ of \cite{Wang09}, we have
\begin{equation}\label{below}
\|\theta\|_{L^\infty(\Omega)}\leq C\|\theta\|_{\Phi_0^k(\Omega)},
\end{equation}
where $C$ depends only on $n,k$,\, and diam$(\Omega)$.

Since $\theta$ satisfies \eqref{k-solution}, we have 
\begin{equation}\label{above}
\|\theta\|_{\Phi_0^k(\Omega)}^{k+1}=\left\lvert\int_\Omega \theta S_k(D^2\theta)\,dx\right\lvert=\frac{1}{k-1}\int_\Omega\theta^2\,dx\leq \frac{\mbox{vol}(\Omega)
}{k-1}\|\theta\|_{L^\infty(\Omega)}^2.
\end{equation}

Combining \eqref{above} and \eqref{below}, we obtain
\[
\|\theta\|_{L^\infty(\Omega)}\leq C\left(\frac{\mbox{vol}(\Omega)
}{k-1}\right)^{1/(k+1)}\|\theta\|_{L^\infty(\Omega)}^{2/(k+1)},
\]
which is equivalent to 
\[
\|\theta\|_{L^\infty(\Omega)}\leq C\left(\frac{\mbox{vol}(\Omega)
}{k-1}\right)^{1/(k-1)},
\]
where $C$ depends only on $n,k$,\, and diam$(\Omega)$.

On the other hand, let $e(x)$ be the unique negative solution of 
\begin{equation}\label{k-torsion}
S_k(D^2e)=1\; \mbox{in }\; \Omega,\; e=0\; \mbox{on }\; \partial\Omega
\end{equation}
and define $w=c\,e$, where $c=\left(\|\theta\|_{L^\infty(\Omega)}/(k-1)\right)^{1/k}$. Then 
\[
S_k(D^2w)=c^kS_k(D^2e)=\|\theta\|_{L^\infty(\Omega)}/(k-1)\geq |\theta|/(k-1)=S_k(D^2\theta)\;\; \mbox{in}\;\;\Omega. 
\]
By the comparison principle for elliptic equations, we conclude that $\theta\geq w=c\,e$ on $\overline{\Omega}$. Hence 
\begin{equation}\label{bound-e}
\|\theta\|_{L^\infty(\Omega)}\leq c\|e\|_{L^\infty(\Omega)}.
\end{equation}
\item[b)] Case $2\leq k<n/2$. To obtain a bound on the $L^\infty$ norm of $\theta$, we apply a result on uniform estimates for $k$-admissible solutions of $k$-Hessian equations as follows. Set $\psi\equiv 1$ and $p=n/k$ in Theorem 2.1 of \cite{ChWa01}. Then the solution $e$ of \eqref{k-torsion} satisfies $\|e\|_{L^\infty(\Omega)}\leq C(\mbox{vol}(\Omega))^{1/n}$, where $C$ is a constant depending only on $n,k$, and the volume $\mbox{vol}(\Omega)$. Combining the last two inequalities that involve the functions $\theta$ and $e$, we obtain
\[
\|\theta\|_{L^\infty(\Omega)}\leq C\left[\frac{(\mbox{vol}(\Omega))^{1/n}}{(k-1)^{1/k}}\right]^\frac{k}{k-1},
\]
where $C$ is a constant depending only on $n,k$, and the volume $\mbox{vol}(\Omega)$.

\item[c)] Case $k=n/2$, $n$ even ($n\neq 2$). The same estimate on the solution of \eqref{k-torsion} holds, but in this case the constant $C$ depends on diam$(\Omega)$ (see the comments that follow the proof of Theorem 2.1 in \cite{ChWa01}). Thus, choosing $\psi$ and $p$ as above, we obtain the estimate
\[
\|\theta\|_{L^\infty(\Omega)}\leq C \left[\frac{4\mbox{vol}(\Omega)}{(n-2)^2}\right]^{1/(n-2)},
\] 
where $C$ is a constant depending only on $n$ and the diameter diam$(\Omega)$. 
\end{itemize}

\noindent The following proof of formula \eqref{formula} is the same as that in \cite{Oliker91} for the case $k=n$. In the most general case under consideration, only the $k$-homogeneity of $S_k(D^2u)$ is required. Let $\tilde{u}_s(t,x)=T(t)\tilde{\theta}(x)$ be a separable solution of \eqref{parak-Hessian} with $T$ in $C^\infty[0,\infty)$ and $\tilde{\theta}$ $k$-admissible and in $\Phi_0^k(\Omega)$. Then $T$ is given by \eqref{Te} and $\tilde{\theta}$ satisfies \eqref{k-solution} for some $\tilde{\lambda}<0$. Let $\tilde{c}$ be a positive constant (to be chosen later) and $\tilde{\theta}_1=\tilde{c}\,\tilde{\theta}$. Then $S_k(D^2\tilde{\theta}_1)=\tilde{c}^k S_k(D^2\tilde{\theta})=\tilde{c}^k\,\tilde{\lambda}\,\tilde{\theta}=\tilde{c}^{k-1}\,\tilde{\lambda}\,\tilde{\theta}_1$. Now choose $\tilde{c}$ so that $\tilde{c}^{k-1}\,\tilde{\lambda}=-\frac{1}{k-1}$. Then $\tilde{\theta}_1$ satisfies \eqref{k-solution} and, because of the uniqueness of the solution of \eqref{k-solution}, $\tilde{\theta}_1$ agrees with $\theta$. Put $c=\tilde{c}^{-1}=[-\tilde{\lambda}(k-1)]^{1/(k-1)}$. Then, using \eqref{Te},
\begin{eqnarray*}
T(t)\tilde{\theta}(x)&=&\tilde{c}^{-1}T(t)\theta(x)\\
&=&c\,\left[T^{1-k}(0)-(k-1)\tilde{\lambda}t\right]^{1/(1-k)}\theta(x)\\
&=&\left\{[c\,T(0)]^{1-k}+t\right\}^{1/(1-k)}\theta(x)\\
&=&u_s(t,x)\left\{\frac{1+t}{[c\,T(0)]^{1-k}+t}\right\}^{1/(k-1)}.
\end{eqnarray*}
Formula \eqref{formula} now follows as well.
\end{proof}

\begin{remark}
We claim that, when $\Omega$ is a ball, the semi-explicit bounds on the $\sup_\Omega|\theta(x)|$ of $\theta$ obtained in Theorem \ref{separablesolution} can be given in closed-form. To see this, we need to write the $k$-Hessian operator in radial form, in which case the equation in \eqref{teta} takes the one-variable form
\begin{equation}\label{radialHessian}
c_{n,k}\,r^{1-n}\left(r^{n-k}(\theta')^k\right)'=-\frac{1}{k-1}\,\theta,\;\; r>0.
\end{equation}
Here $r=|x|,\, '=d/dr$ and $c_{n,k}$ is defined by $c_{n,k}=\binom{n}{k}/n$. To prove the claim, let $\Omega=B$ be a ball of radius $R>0$ and let $N=\sup_B |D\theta|$. Then we have $|\theta|\leq Nd$ where $d=\mbox{dist}\,(x,\partial B)$.
Since $\theta$ satisfies \eqref{k-solution}, we have
\[
\int_B S_k(D^2\theta)\,dx=-\frac{1}{k-1}\int_B \theta\,dx=\frac{1}{k-1}\int_B \lvert\theta\rvert\,dx\leq\frac{\mbox{vol}(B)}{k-1}\sup_B|\theta(x)|.
\]
Now from Theorem \ref{radialasymptotic}, $\theta$ is radially symmetric and by \eqref{radialHessian} convex in the radial direction $r$. Hence their gradient attains its maximum on the boundary $r=R$ of $B$. Consequently,
\[
\int_B S_k(D^2\theta)\,dx=\omega_n\,c_{n,k}\,R^{n-k}(D\theta (R))^k\geq\omega_n\,c_{n,k}\,\frac{R^{n-k}(\sup_B|\theta(x)|)^k}{(2R)^k},
\]
where $\omega_n$ is the volume of unit ball in $\RR^n$. Therefore, from the last two inequalities, we conclude that
\begin{equation}\label{explicitball}
\sup_B|\theta(x)|\leq\left[\frac{(2R)^k\mbox{vol}(B)}{\omega_n R^{n-k}(k-1)c_{n,k}}\right]^{1/(k-1)}. 
\end{equation}
Note that, for $k=n$, this bound agrees with that given in \cite{Oliker91} when the domain $\Omega$ there is also a ball. Further, the estimate \eqref{explicitball} can be improved using \eqref{bound-e} and the fact that \eqref{k-torsion} has the unique solution
\[
e(x)=c\left(|x|^2-R^2\right),
\]
where $c=\frac{1}{2}(nc_{n,k})^{-1/k}$.
\end{remark}

\begin{proof}[Proof of Theorem \ref{asymptotic}]
We point out that the proof given here is basically the same proof given in \cite{Oliker91} in the case $k=n$, {\it mutatis mutandis}. More interestingly in our context, we can prove the same theorem for $k$-admissible solutions and strictly $(k-1)$-convex domains in the full range $2\leq k\leq n$. Here some technical and classical results from the theory are needed, but the underlying ideas are essentially the same. Let $u(t,x)$ be a solution of \eqref{parabolick-Hessian}. The asymptotic behavior of $u(t,x)$ will be determined by constructing sub- and supersolutions with the use of separable solutions (which act as barriers for solutions of \eqref{parak-Hessian}) as in Theorem \ref{separablesolution}. 

Put $\underline{T}(t)=\left[\underline{T}^{1-k}(0)+t\right]^{-1/(k-1)}$, where we choose $\underline{T}(0)$ so that $\underline{T}(0)\theta(x)\leq u(0,x)$. Further, we put
$\underline{u}(t,x)=\underline{T}(t)\theta(x)$, where $\theta$ satisfies \eqref{k-solution}. Then $\underline{u}(t,x)$ is a separable solution of the equation
\[
\underline{u}_t=S_{k}(D^2\underline{u})\; \mbox{ in}\; (0,\infty)\times\Omega\; \mbox{ and}\; \underline{u}(0,x)=\underline{T}(0)\theta(x).
\]
Put $\tilde{u}(t,x)=\underline{u}(t,x)-u(t,x)$. Then 
\begin{equation}\label{equality1}
\tilde{u}_t=S_{k}(D^2\underline{u})-S_{k}(D^2u)\; \mbox{ in}\; (0,\infty)\times\Omega.
\end{equation}
For each $\tau\in [0,1]$, the function $u_\tau (t,x)=\tau\,\underline{u}(t,x)+(1-\tau)u(t,x)$ is a $k$-admissible function since linear combinations of $k$-admissible functions with non-negative coefficients are also $k$-admissible. See \cite[Lemma 2.3]{TrWa97}. Further, for each fixed $t$ and $x$ in $\Omega$, we have
\[
L(\tilde{u})\equiv S_{k}(D^2\underline{u})-S_{k}(D^2u)=\sum_{i,j}a_{ij}\tilde{u}_{ij},
\]
where $a_{ij}=\int_0^1 S_k^{ij}(D^2 u_{\tau ij})\,d\tau$. Since $u_\tau$ is $k$-admissible, $\left\{S_k^{ij}(D^2 u_{\tau ij})\right\}$ is a positive definite matrix (see {\it e.g.} \cite{Wang09}) on any subset of $(0,\Lambda]\times\Omega$ for any $\Lambda<\infty$, that is, $L$ is elliptic.
Now we rewrite \eqref{equality1} as
\begin{equation}\label{equality2}
\tilde{u}_t=L(\tilde{u})\; \mbox{ in}\; (0,\infty)\times\Omega.
\end{equation}
Note also that
\begin{equation}\label{initialandboundary}
\tilde{u}(0,x)\leq 0\; \mbox{ in}\; \overline{\Omega}\; \mbox{ and}\; \tilde{u}(t,x)=0\; \mbox{ in}\; [0,\infty)\times\partial\Omega. 
\end{equation}
We consider now the differential equality \eqref{equality1} with the initial and boundary conditions \eqref{initialandboundary} for $t\leq\Lambda$ with any $\Lambda<\infty$. It follows from the classical maximum principle that $\tilde{u}(t,x)=\underline{u}(t,x)-u(t,x)\leq 0$ in $[0,\infty)\times\overline{\Omega}$. Consequently, 
\begin{equation}\label{lowerbound}
\left\{(1+t)^{1/(k-1)}\left[\underline{T}^{1-k}(0)+t\right]^{-1/(k-1)}-1\right\}\theta(x)\leq (1+t)^{1/(k-1)}u(t,x)-\theta(x).
\end{equation}
A separable supersolution $\overline{u}(t,x)$ is obtained as follows. Let
\[
\overline{T}(t)=\left[\overline{T}^{1-k}(0)+t\right]^{-1/(k-1)},
\]
where $\overline{T}(0)$ is such that $\overline{T}(0)\theta(x)\geq u(0,x)$, with $\theta$ satisfying \eqref{k-solution}. Put $\overline{u}(t,x)=\overline{T}(t)\theta(x)$. The function $\overline{u}$ satisfies the equation
\[
\overline{u}_t=S_{k}(D^2\overline{u})\; \mbox{ in}\; (0,\infty)\times\Omega\; \mbox{ and}\;\; \overline{u}(t,x)=0\,\mbox{ on}\; [0,\infty)\times\partial\Omega.
\]
Then $u_t-\overline{u}_t=S_{k}(D^2u)-S_{k}(D^2\overline{u})=L(u-\overline{u})$. Applying the maximum principle and taking into account the inequality $\overline{T}(0)\theta(x)\geq u(0,x)$, we conclude that $u(t,x)\leq \overline{u}(t,x)$ in $[0,\infty)\times\overline{\Omega}$.
From this we obtain 
\begin{equation}\label{upperbound}
(1+t)^{1/(k-1)}u(t,x)-\theta(x)\leq\left\{(1+t)^{1/(k-1)}\left[\overline{T}^{1-k}(0)+t\right]^{-1/(k-1)}-1\right\}\theta(x).
\end{equation}
As in \cite{Oliker91}, it is convenient to consider the function $F(s,t):(0,S]\times [0,\infty)\rightarrow (0,\infty),\; S<\infty$, given by
\[
F(s,t)=\left[\frac{1+t}{s+t}\right]^{1/(k-1)}\equiv\left[1-\frac{s-1}{1+t}\,\frac{1+t}{s+t}\right]^{1/(k-1)}.
\]
The function $(1+t)/(s+t)$ is nonincreasing in $t$ when $s\leq 1$ and nondecreasing when $s\geq 1$. Therefore for all $t\geq 0$ when $s\leq 1$
\[
F(s,t)\leq \left[1+\frac{1-s}{s(1+t)}\right]^{1/(k-1)}\leq 1+\frac{1}{k-1}\,\frac{1-s}{s(1+t)}
\]
and, when $s\geq 1$,
\[
F(s,t)\geq 1-\frac{s-1}{1+t}\,\frac{1+t}{s+t}\geq 1-\frac{s-1}{1+t}.
\] 

On the other hand, from the construction of $\underline{T}(0)$ it is clear that we can assume that $\underline{T}(0)\geq 1$. Therefore, $F\left(\underline{T}^{1-k}(0),t\right)\leq 1+(1+t)^{-1}C_1$, where $C_1=\left[1-\underline{T}^{1-k}(0)\right]\frac{\underline{T}^{1-k}(0)}{k-1}$.
Similarly, in deriving \eqref{upperbound} we may assume that $\overline{T}(0)\leq 1$. Then
\[
F\left(\overline{T}^{1-k}(0),t\right)\geq 1-(1+t)^{-1}C_2,
\]
where $C_2=\overline{T}^{1-k}(0)-1$.
Combining \eqref{lowerbound} and \eqref{upperbound}, we obtain for all $t\geq 0$ and $x\in\overline{\Omega}$ that
\begin{equation}\label{lubound}
C_1(1+t)^{-1}\theta(x)\leq (1+t)^{1/(k-1)}u(t,x)-\theta(x)\leq -C_2(1+t)^{-1}\theta(x).
\end{equation}
Let $C=\max\{C_1,C_2\}\sup_{\Omega}\left|\theta\right|$. Then \eqref{lubound} implies \eqref{bounded}.
\end{proof}
\begin{remark}
Similarly to \cite{ArPe81}, the sharpness of the estimate \eqref{lubound} is easily seen by considering the function $u(t,x)=(s+t)^{-1/(k-1)}\theta(x)$ for any $s\in (0,\infty)$.
\end{remark}

\begin{proof}[Proof of Theorem \ref{radialasymptotic}]
Consider the function $\Theta(x)=\theta(Ux)$, where $U$ is an arbitrary rotation. Then, recalling that the $k$-Hessian operator is invariant under rotations of coordinates, we have $S_k(D^2\Theta(x))=-\frac{1}{k-1}\Theta(x)$ and $\Theta(x)=0$ on $\partial\Omega$. It is also easy to see that $\Theta$ is $k$-admissible. Therefore, $\Theta(x)=\theta(Ux)$ is a solution of \eqref{k-solution}. But the solution of \eqref{k-solution} is unique. Hence, $\theta(x)=\theta(|x|)$. Combining this statement with the assertion \eqref{bounded} in Theorem \ref{asymptotic}, the proof is complete.
\end{proof}

\begin{remark}
Let $u$ be a solution of equation in \eqref{parak-Hessian}. Define $M(t,\Omega)=\int_\Omega \vert u(t,x)\vert dx$. We call this quantity the {\it total mass} of $u$ at time $t$. Note that the mass of a solution of Problem \eqref{parak-Hessian} decreases in time. Indeed, integration of the equation with respect to $x$ gives $\frac{d}{dt}M(t,\Omega)=-\int_\Omega u_t(t,x)dx=-\int_\Omega S_k(D^2u)dx<0$. We will see in the next section that this is not the case for the family of self-similar solutions given in \eqref{k-Baren} when $\Omega=\RR^n$.   
\end{remark}

\section{$k$-Barenblatt solutions}\label{k-Barenblatt}
In this section we will derive the compactly supported family of mass conserving $k$-Barenblatt solutions given in \eqref{k-Baren} for the equation 
\begin{equation}\label{wrt}
u_t=(-1)^{k-1} S_k(D^2 u).
\end{equation} 
Thus, we are looking for a positive solution of the above evolution equation with constant mass in $\RR^n$, that is
\begin{equation}\label{masscondition}
\int_{\RR^n} u(t,x) dx = M>0,\,\, \mbox{for all}\,\, t>0.
\end{equation}
Due to the homogeneity of \eqref{wrt}, we will actually look for a self-similar solution $u$ to \eqref{wrt} of the form:
\begin{equation}\label{self:0}
u(t,x) = t^{-\alpha}\theta(\xi),\; \xi = \frac{x}{t^{\beta}},\, t>0,\; x\in \RR^n, 
\end{equation}
for some profile $\theta$ and the exponents $\alpha$ and $\beta$ to be determined. Inserting \eqref{self:0} into the left-hand side of \eqref{wrt}, we have
\begin{align*}
u_t& = -\alpha t^{-\alpha-1}\theta(\xi) + t^{-\alpha} \frac{d\theta}{d\xi}\cdot\frac{d\xi}{dt}\\
& =-\alpha t^{-\alpha-1}\theta(\xi)+t^{-\alpha}\nabla_{\xi} \theta (\xi)\cdot (-\beta)t^{-\beta-1}x\\
& =  t^{-\alpha-1}(-\alpha\theta(\xi) -\beta\nabla_{\xi} \theta(\xi)\cdot\xi).
\end{align*}

Inserting \eqref{self:0} into the right-hand side of \eqref{wrt} (omitting the scalar factor) we have
\begin{align*}
S_k(D^2u)& = t^{-k\alpha-2k\beta}S_k(D^2\theta(\xi))\\
& =t^{-k(\alpha+2\beta)}S_k(D^2\theta(\xi)).
\end{align*}
Then, from the condition $\alpha (k-1)+2k\beta=1$ (self-similarity condition), we get the following profile equation
\begin{equation}\label{k-profile}
\alpha\theta(\xi)+\beta\nabla_{\xi} \theta(\xi)\cdot\xi=(-1)^k S_k(D^2\theta(\xi)).
\end{equation}

We also have from \eqref{masscondition}
\[
M=\int_{\RR^n} u(t,x)\,dx = \int_{\RR^n} t^{-\alpha}\theta\left(\frac{x}{t^{\beta}}\right)dx = t^{n\beta-\alpha}\int_{\RR^n} \theta(\xi)\,d\xi
\]
(it is assumed that $\theta\in L^1(\RR^n)$), which yields $n\beta-\alpha= 0$ (mass-preserving condition). Solving the relations between the similarity exponents $\alpha$ and $\beta$ we obtain
$\alpha=\frac{n}{n(k-1)+2k}$ and $\beta=\frac{1}{n(k-1)+2k}$. 

Now let $\theta$ be a radially symmetric function, say $\theta=\theta(r),\, r=\abs{\xi}\geq 0$. Then the governing equation \eqref{k-profile} takes the form
\begin{equation}\label{radial}
\alpha\theta(r)+\beta r\theta'(r)=(-1)^kc_{n,k}\,r^{1-n}(r^{n-k}(\theta'(r))^k)',\; r>0,
\end{equation}
with the symmetry condition $\theta'(0)=0$. From this and the equality $\alpha=n\beta$, the equation \eqref{radial} can be integrated once (fortunately) and then simplified as
\begin{equation}\label{radialprofile}
\beta\theta(r)=(-1)^{k}c_{n,k}\,r^{-k}(\theta'(r))^k,\; r>0;\;\theta'(0)=0.
\end{equation}
We observe that, when $k=1$ in \eqref{radialprofile}, an explicit integration shows that $\theta(r)=Ce^{-\frac{r^2}{4}}$, where $C$ is a positive constant. Thus from \eqref{self:0} we recover the Gaussian function of the classical heat equation. Now let $k>1$. A necessary condition for the existence of a solution with the required properties is that the profile $\theta$ be decreasing. Thus integrating \eqref{radialprofile} we have
\begin{equation}\label{theta}
\theta(r)=\left(C-\frac{k-1}{k}\left(\frac{\beta}{c_{n,k}}\right)^\frac{1}{k}\frac{r^2}{2}\right)_+^{\frac{k}{k-1}},\;\; r\geq 0.
\end{equation}
Finally, putting $\gamma=\frac{k-1}{2k}\left(\frac{\beta}{c_{n,k}}\right)^\frac{1}{k}$ and inserting $\theta(r)$ in \eqref{self:0}, we obtain \eqref{k-Baren}.

Note that the positive constant $C$ in \eqref{theta} may easily be put in correspondence with the mass of the solution, $C=C(M)$, by \eqref{masscondition}. In fact, introducing the constant $r_0=\sqrt{\frac{2C}{\gamma}}$, the self-similar solution with constant mass has the explicit form
\begin{equation}\label{ssm}
u(t,x)=t^{-\frac{n}{n(k-1)+2k}}\left[\frac{k-1}{4k[c_{n,k}(n(k-1)+2k)]^\frac{1}{k}}\left(r_0^2-\frac{\abs{x}^2}{t^\frac{2}{n(k-1)+2k}}\right)_+\right]^\frac{k}{k-1},
\end{equation}
where
\[
r_0(M)=\left\{\pi^{-\frac{n}{2}}\left(\frac{4k}{k-1}\right)^\frac{k}{k-1}[c_{n,k}(n(k-1)+2k)]^\frac{1}{k-1}\frac{\Gamma\left(\frac{n}{2}+\frac{2k-1}{k-1}\right)}{\Gamma\left(\frac{2k-1}{k-1}\right)}\,M\right\}^\frac{k-1}{n(k-1)+2k}
\]
and where $\Gamma(\cdot)$ is the Gamma-function.

We have the following properties of the self-similar solutions given in \eqref{k-Baren}:
\begin{itemize}
\item[$\bullet$] supp\,$U_C(t,\cdot)\subseteq B\left(0,\,t^\beta\left[\frac{2k}{k-1}\left(\frac{c_{n,k}}{\beta}\right)^\frac{1}{k}C\right]^\frac{1}{2}\right)$.

\item[$\bullet$] Finite propagation speed.

\item[$\bullet$] Mass conservation.

\item[$\bullet$] $\lim_{t\rightarrow 0^+}U_C(t,x)=M\delta_0(x)$, where $\delta_0(x)$ is Dirac's delta function concentrated at 0.

\item[$\bullet$] Everywhere, except on the degeneracy surface $[0,\infty)\times\left\{\abs{x}=t^\beta\left[\frac{2k}{k-1}\left(\frac{c_{n,k}}{\beta}\right)^\frac{1}{k}C\right]^\frac{1}{2}\right\}$, it is classical (and infinitely differentiable).
\end{itemize}

\bibliographystyle{plain}
\bibliographystyle{apalike}
\bibliography{kHessianbib}
\end{document}